\documentclass[12pt]{amsart}
\usepackage{amssymb}
\usepackage{comment}

\theoremstyle{plain}
\newtheorem{thm}{Theorem}[section]

\newtheorem{prop}[thm]{Proposition}
\newtheorem{lemma}[thm]{Lemma}
\newtheorem{defn}{Definition}
\newtheorem{remark}{Remark}

\def\L{\mathcal L}

\def\CC{\mathbb C}

\def\HH{\mathbb H}
\def\RR{\mathbb R}
\def\QQ{\mathbb Q}

\def\Pr{\rm Pr}
\def\L{{\bf L}}
\def\I{{\bf I}}

\def\H{\mathcal{H}}

\usepackage{comment}
\title[Palindromic Generators]{Discreteness Criteria and the Hyperbolic Geometry of
Palindromes} 
\author{Jane Gilman and Linda Keen }
\address{Department of  Mathematics, Rutgers University, Newark, NJ
 07079} \email{gilman@rutgers.edu}
\address{Department of  Mathematics,  Lehman College and Graduate
Center,  CUNY, Bronx, 10468} \email{Linda.keen@lehman.cuny.edu}
\thanks{PSC-CUNY}
\thanks{Rutgers Research Council and Yale University}
\begin{document}
\maketitle
\begin{abstract}

We consider non-elementary representations of two generator free
groups in $PSL(2,\mathbb{C})$, not necessarily discrete or free, $G
= \langle A, B \rangle$. A word in $A$ and $B$, $W(A,B)$,  is a
palindrome if it reads the same forwards and backwards.  A word in a
free group is {\sl primitive} if it is part of a minimal generating
set. Primitive elements of the free group on two generators can be
identified with the positive rational numbers. We study the geometry
of palindromes and the action of $G$ in $\HH^3$ whether or not $G$
is discrete. We show that there is a {\sl core geodesic} $\L$ in the
convex hull of the limit set of $G$ and use it to prove three
results:  the first is that there are well defined maps from the
non-negative rationals and from the primitive elements to $\L$; the
second is that $G$ is geometrically finite if and only if the axis
of every non-parabolic palindromic word in $G$ intersects $\L$ in a
compact interval; the third is   a description of the relation of
the pleating locus of the convex hull boundary
  to the core geodesic and to palindromic elements.

\end{abstract}

\maketitle

\section{Introduction}

 In this paper we are interested in
the
  interplay between algebraic and
geometric properties of non-elementary $PSL(2,\CC)$ representations
of a free group on two generators. A word in a rank two free group
is primitive if it is part of some two generator generating set. A
word is a palindrome if it reads the same forwards and backwards.
The connection between primitive words and palindromes in a free
group of rank two is well understood. While the image of a
palindrome is a palindrome under a representation, if the image is
not a free group, the concept of primitive element does not make
sense. The images of a pair of  primitive generators for the free
group, however,    still generate the group and the palindromic
property remains important.

From the algebraic point of view, we know that primitive words in
the free group of rank two are indexed by the rationals and that
each is either conjugate to a palindrome or a product of palindromes
\cite{GKenum}. We represent the groups as subgroups of $PSL(2,\CC) =
SL(2,\CC)/\{ \pm Id \}$, so that they are subgroups of the isometry
group of   hyperbolic three space and define the notion of a {\em
core geodesic} for the group. This is a geodesic in $\HH^3$ about
which the convex hull of the limit set of $G$ is symmetric.  We
obtain a map, {\sl  the $\Pi$-map}, from the axes of the
representations of primitive words, and hence from the rationals to
this core geodesic (Theorems \ref{thm:firstPi} and
\ref{thm:secondPi}).

We use the core geodesic to obtain our main result,  a necessary and
sufficient discreteness condition for $G$. We prove that if the
representation is
  geometrically finite,
the axis of every non-parabolic palindromic word intersects the core
geodesic
  in a compact subinterval and conversely,
if the axis of every non-parabolic palindrome intersects the core
geodesic in a compact subinterval, the representation is discrete
(Theorems~\ref{thm:gfif},~\ref{thm:gfifp} and ~\ref{thm:conv}).

When the representation is free, discrete  and geometrically finite,
we consider the pleating locus of the convex hull boundary and we
show that either the pleating locus is not maximal or there are
leaves intersecting the core geodesic orthogonally and these are
limits of axes of palindromes (Theorem \ref{thm:pleats}).   One can
also define the notion of a pleating locus for the convex hull
boundary when the representation is discrete but contains elliptic
elements, and thus represents an orbifold rather than a manifold. We
prove that Theorem \ref{thm:pleats} applies in that case as well.

The organization of the paper is as follows: sections
\ref{section:prelim}, \ref{section:hyp3} and \ref{section:primwds}
are preliminary in nature. We fix notation and review terminology
including the definition of the core geodesic.  Section
\ref{section:geometric} begins with the key lemma which establishes
the map from  palindromic axes to the core geodesic
(Lemma~\ref{prop:palifforthog}). In
section~\ref{section:discreteness} the necessary discreteness
condition is proved first for groups without parabolics and then
extended to groups with parabolics. The proof of sufficiency then
follows. Section~\ref{section:pleating} begins by carefully
extending notation and definitions of pleated surfaces to groups
with parabolics and elliptics. It ends with a description of the
relation of the pleating locus of the convex hull boundary
  to the core geodesic (Theorem \ref{thm:pleats}).

\section{Preliminaries: Notation and Terminology}
\label{section:prelim}

We recall some terminology and fix some notation.
\subsection{Notation}

 We are interested in  conjugacy classes of    representations
$\rho$ of the free group $F= \langle a, b \rangle$ into
$PSL(2,\CC)$.  Since, however, the properties we are concerned with
are conjugacy invariant, we work with specific representations. That
is, we can pick a representation  $G= \rho(F)$ with  $\rho(a) = A$
and $\rho(b) = B$  in the conjugacy  class  by
 by choosing specific matrices in $SL(2,\CC)$ for $A$ and $B$.

  In order to simplify
the exposition, we use the same notation for  a matrix $A$ in
$SL(2,\CC)$, its equivalence class  in $PSL(2,\CC)$    and the
isometry of $\HH^3$ it induces. If the group with which we are
working   is not clear from the context, we will make it explicit.

\subsection{Primitive Elements}

If $F$ is the free group on $n$ generators, an element $W \in F$ is
termed {\sl primitive} if it can be extended to be part of a minimal
generating set, that is, a set of $n$ generators. If we are talking
about a free group of rank two, an element $b$ is a {\sl primitive
associate} for a primitive element $a$ if $F= \langle a,b \rangle$.
A primitive element $a$ has a set of primitive associates that are
described in {\cite{GKwords, GKenum}} and the unordered pair $a$ and
$b$ is called a {\sl primitive pair}.

\subsection{Palindromic condition}

In any group, a {\sl palindrome} is a word in the generators that
reads the same forwards and backwards. We let $\underleftarrow{W}$
denote the word read backwards. For any two generator group,
$\langle A, B \rangle$, if   $W(A,B)$ is a word,
$\underleftarrow{W}=(W(A^{-1},B^{-1}))^{-1}$. The word $W(A,B)$ is a
palindrome if and only if  $W(A,B)= {\underleftarrow{W}}(A,B)$. Note
that this definition depends on the generating set.

For representations $\rho(F)=G \subset PSL(2,\CC)$ we are interested
both in the full set of palindromic words and the set of palindromic
generators.  In the free group on two generators, the palindromic
primitives are well understood;    this is summarized in section
\ref{section:primwds}.

\section{Preliminaries: Hyperbolic three space} \label{section:hyp3}

Hyperbolic three space $\HH^3$ is the set of points $\{(x,y,t)\in
\RR^3 \;|\; t>0\}$ together with the hyperbolic metric $d_{\HH}$. In
what follows we use the notation and terminology of Fenchel
\cite{Fench}.

 The boundary of $\HH^3$ is (isomorphic to) the Riemann
sphere ${\hat{\CC}}$ and we write $\partial\HH^3 = {\hat{\CC}}$ and
${\overline{\HH^3}} = \HH^3 \cup {\hat{\CC}}$. Two distinct points
$u$ and $v$ in ${\overline{\HH^3}}$ determine a unique geodesic in
$\HH^3$ which we denote by $[u,v]$ and its closure,
$\overline{[u,v]}$, intersects $\partial{\HH^3}$ in two points, $x$
and $y$ which are the {\sl ends} of $[u,v]$. We include in our
geodesics, {\sl degenerate lines}; that is, if $x=y$ we consider
{\sl the degenerate line} $[x,x]$,  which is a point on
$\partial\HH^3$.

  The trace of an element $X \in SL(2,\CC)$  is well defined.  We
  can use this to classify the image of $X$ in $PSL(2,\CC)$.  Note
  that the classification is independent of the pull-back.

\subsection{Axes and the Core  $A$-$B$ Normal.} \label{section:core}

{\sl Every} element $X \in PSL(2,\CC)$ fixes a geodesic called its
{\sl axis} and denoted by $Ax_X$. If $X$ is loxodromic or elliptic
with fixed points $x$ and $y$ in $\partial\HH^3$, then $Ax_X =
[x,y]$; if $X$ is parabolic with fixed point $p$, then $Ax_X =
[p,p]$. In all cases, for any pair $A,B \in PSL(2,\CC)$ there is a
unique perpendicular from $Ax_A$ to $Ax_B$ regardless of whether the
axes are proper or degenerate lines (see 2.6 of \cite{Fench}).

 In earlier
papers, \cite{GKwords,GKenum} we denoted this unique perpendicular,
or normal, to the axes of the generators $A$ and $B$ by $\L$ or
$\L_{AB}$ and called it the $\L-$line. The $\L-$line depends, of
course, upon the choice of generators. It will play an important
role in what follows.

\begin{defn} For a fixed pair $A$ and $B$ we call the common normal,
the {\sl core geodesic}.
\end{defn}

This terminology will be justified later.

\smallskip

We recall that a  subgroup of $PSL(2,\CC)$ is elementary if  there
is some point in $\hat\CC$ whose orbit under the group is
  finite. If the axes of the two generators $A$ and $B$ share exactly one end,
then the common normal reduces to a point on the boundary and the
group $G = \langle A,B \rangle$ is elementary.

\smallskip
In what follows we restrict ourselves to non-elementary groups $G$.
This excludes the case of a degenerate core geodesic.

\subsection{Half-turns} \label{section:halfturns}

A {\sl half-turn} about a geodesic $X$,
is denoted by $H_X$ and is the unique orientation preserving
involution fixing $X$ point-wise.

If $p \not\in X$ is any point in ${\overline{\HH^3}}$, there is a
unique perpendicular from $p$ to $X$;    $H_X$ maps this to the
perpendicular from $H_X(p)$ to $X$ so that $H_X$ fixes the geodesic
$[p,H_X(p)]$. If $Q = [x,y]$ with $H_X(Q) = Q$, then $Q= [x,H_X(x)]$
and $Q$ is perpendicular to $X$. Conversely,  If $Q=[x,y]$ is any
geodesic perpendicular to $X$, then we must have $y=H_X(p)$ and $Q =
[x,H_X(x)]= H_X(Q)$.

\subsection{The core geodesic and the double altitude} \label{section:hex}

In our theorems a particular element of $PSL(2,\CC)$  determined
algebraically by $A$ and $B$ plays an important role. We give the
geometric construction of its axis here and the algebraic
characterization in section~\ref{sec:mapPi}.

 Choose   lines $\L_A$ and $\L_B$   so
that $A$ and $B$ factor as half-turns: $A = H_{\L_A} \circ H_{\L}$
and $B = H_{\L} \circ H_{\L_B}$. The axis of $AB$ is then the common
perpendicular to $\L_A$ and $\L_B$ and the six geodesics $Ax_A$,
$\L$, $Ax_B$, $\L_B$, $Ax_{AB}$, $\L_A$ determine  a standard
right-angled hexagon $\H$ (\cite{Fench}).   The lines $\L_A$ and
$\L_B$ are thus also core geodesics (for the pairs of generators
$(A, AB)$ and $(B, AB)$). The {\sl sides} of the hexagon lie
alternately on axes and core geodesics.   The hexagon $\H$ can be
formed whether or not $G$ is discrete.

The common perpendiculars to opposite pairs of sides of the hexagon
are the altitudes. For example, there is an altitude $\mathcal A$
from $Ax_{AB}$ to $\L$. Note that $Ax_{BA} = Ax_{A^{-1}B^{-1}}$ and
$H_{\L}ABH_{\L}= A^{-1}B^{-1}$ so that the altitude from $Ax_{BA}$
to $\L$ in $H_{\L}(\H)$ and  $\mathcal A$ lie on a common line, the
  perpendicular from $Ax_{AB}$ to $Ax_{BA}$ which we denote by by $N_{AB}$ . This is the {\em
double altitude} of the doubled hexagon, $\H \cup H_{\L}(\H)$.

\subsection{Geometrically primitive elliptics} \label{section:primell}
In any cyclic group, an element is {\sl primitive} if it generates
the group. This definition obviously applies to elements in the
cyclic group generated by an elliptic element of finite order in
$PSL(2,\CC)$. Such an elliptic element acts on $\partial\HH^3$ as a
rotation by a certain angle. A primitive elliptic element that
rotates through a {\sl minimal} angle is termed {\sl geometrically
primitive}.

\section{Preliminaries: Primitive Words and Rational Numbers} \label{section:primwds}

We recall facts about primitive words and rationals in two generator
groups. As above, $F=\langle a,b \rangle$ is the free group and $G$
is the image of a representation into $SL(2,\RR)$ 
(see \cite{GKenum} for fuller details).

We use $\Pr_F$ to denote the set of primitive elements in $F$ and
$\Pr_G$ to denote the image of $\Pr_F$ in $G$. Since $G$ is not
 necessarily free, the term primitive may not apply to an element of
$\Pr_G$.  Nonetheless, an element of $\Pr_G$ is a generator of a two
generator group and has a set of associates with which it generates
the full group. If an element of $Pr_G$ is elliptic of finite order,
it may not be geometrically primitive.

Conjugacy classes of primitive words in the generators $A,B$ of a
rank two free group are indexed by the non-negative rationals
together with infinity\footnote{Primitive words in the generators
$A^{-1}$ and $B$ are indexed by the non-positive rationals}. We
denote a rational number by $p/q$ where $p$ and $q$ are relatively
prime integers with $q
>0$. There are several useful iteration schemes for choosing
representatives of the conjugacy
classes, see for example, \cite{GKwords,GKenum,KS}.

In the recursive scheme defined in \cite{GKenum},  the words in the
free group $F=\langle a,b \rangle$   are denoted by $e_{p/q}$ and
have the following properties: when $pq$ is even there is always a
unique palindrome in the conjugacy class and this is $e_{p/q}$. When
$pq$ is odd, elements in the $e_{p/q}$ conjugacy can be factored
(possibly in a number of different ways) as the product of
palindromes, but the enumeration scheme determines a unique
factorization of $e_{p/q}$ as a product of a pair of palindromes
that have previously appeared in the scheme.
  Pairs of words $(e_{p/q},
e_{r/s})$ are primitive associates if and only if $|ps-rq|=1$.

  If a representation of $F$ into $PSL(2,\CC)$  is conjugate to a subgroup of
$PSL(2,\RR)$ and is discrete, but not free, the pairs of generators
and the  indexing by non-negative rationals must be slightly
modified to take into account generators that are not geometrically
primitive.

\section{Geometric Implications } \label{section:geometric}
Let us see what geometric implications  there are  for arbitrary
representations $\rho$ of $F=\langle a,b \rangle $ into
  $  PSL(2,\CC)$ whose image $G  = \rho(F)$ is  non-elementary.

  If $w$ is a
word in $F$,  we write $W$ for $\rho(w)$ and use the notation
$\rho(w(a,b)) = W(A,B)$  and $\rho(e_{p/q}) = E_{p/q}$. If $|ps-qr|
=1$, the images
  $E_{p/q}=\rho(e_{p/q})$, $E_{r/s}=\rho(e_{r/s})$  of the primitive associate pair
  $ e_{p/q} , e_{r/s} $ in $F$    generate the group $G$. As usual we denote the core geodesic   by $\L$.

\subsection{Palindromes and the core geodesic} \label{section:palcore}

The following  proposition says that $\L$ is a line of symmetry for
the the convex hull of the limit set of the group and justifies
calling it the core geodesic.  Note that if the group is not
discrete the convex hull is all of $\HH^3$.
 Note also that  the words $W$ in this proposition need not be primitive.

\begin{lemma}\label{prop:palifforthog} Let $G= \langle A,B \rangle$ be a non-elementary
subgroup of $PSL(2,\CC)$.
 The word $W= W(A,B)$ is a palindrome in $A$ and $B$ if and only if the axis of
$W$ is orthogonal to the core geodesic,   $\L$. \end{lemma}

\begin{proof} Let $H_{\L}$ be the half turn about $\L$ as in section \ref{section:halfturns}.
Then $$H_{\L} A H_{\L}=A^{-1} \mbox{ and }   H_{\L} B H_{\L} =
B^{-1}.$$ Thus $$ H_{\L} W(A,B) H_{\L} = W(A^{-1},B^{-1}).$$ If  $W$
is a palindrome,  then $W(A,B)^{-1} = W(A^{-1},B^{-1})$  and the
axis of $W$ is sent to itself under $H_{\L}$ interchanging its fixed
points. It is, therefore, orthogonal to $\L$. On the other hand,
following section~\ref{section:halfturns}, if the axis of $W$ is
orthogonal to $\L$, then $H_{\L}$ sends the axis of $W$ to itself
interchanging the fixed points. Since it sends $A$ and $B$ to their
respective inverses,  $Ax_{W(A,B)} = Ax_{W(A^{-1},B^{-1})}$, so that
$W(A,B) = W(A^{-1},B^{-1})^{-1}$ and $W$  is a palindrome.
\end{proof}

\subsection{Changing generators}

 Let $$\phi_a:\langle a,b \rangle \rightarrow \langle a,ab \rangle,
 \, \, \phi_b:\langle a,b\rangle \rightarrow \langle ba,b \rangle$$
 define automorphisms of $F$.
 Let $c=ab$, $d=ba$.

 Using the enumeration scheme in \cite{GKenum} and the discussion in \cite{GKwords}, it is not difficult
 to see that if $e_{p/q}$ is a palindrome in $a$ and $b$, then either
 $\phi_a(e_{p/q})$ is a palindrome in $a$ and $c$ or $\phi_b(e_{p/q})$ is a palindrome
 in $d$ and $b$. If $e_{p/q}$ is a not palindrome in $a$ and $b$, then both
 $\phi_a(e_{p/q})$ and $\phi_b(e_{p/q})$ are palindromes in the respective new generators.

 Using these facts, it follows that the lines $\L_A$ and $\L_B$ are core
 geodesics for $G$ presented as $G= \langle A,C \rangle$ and $G= \langle
 B,D  \rangle$ respectively.  That is, $W(A,C)$ is a palindrome in
 $A$ and $C$ if and only if the axis of $W$ is orthogonal to the
 $\L_A$ core geodesic and similarly for $\L_B$.

\subsection{The map $\Pi$}\label{sec:mapPi}

Let $X$ and $Y$ be any two words in $G$ and recall that the double
altitude $N_{XY}$ is the common normal to the axes of $XY$ and $YX$.
In what follows we compute with several related objects
 but try not to distinguish between
 them in our notation except when required for clarity. These include a two by two
 non-singular matrix, the same matrix normalized to be in $SL(2,\CC)$,
 and its image in $PSL(2,\CC)$. Further a  matrix determines a geodesic, namely the
 axis that it fixes when acting on  $\HH^3$. If such a matrix is of order two,
 it is the {\sl line  matrix} determining that geodesic.

\begin{lemma}\label{prop:NAB}
 Let $U$ and $V$ be any pair of
  words in the non-elementary group $G$ that are palindromes in the generators
  $A$ and  $B$. Then $N_{UV}$,  the common
orthogonal  to the axes of $UV$ and $VU$,   is orthogonal to $\L$
and  is the axis of a palindrome.  \end{lemma}

\begin{proof} If $M$ is a square matrix, we let $\det{M}$ denote its determinant.
If $X$ and $Y$ are two  elements of $PSL(2,\CC)$, the line matrix determining their common perpendicular
is
 $(XY-YX)/\det(XY-YX)$ (\cite{Fench}).

Let $X=UV$ and $Y= VU$, then the altitude $N_{UV}$ is their common
perpendicular and its line matrix is  $$T_{UV} =
(UVVU-VUUV)/\det(UVVU-VUUV).$$     To see this note that
$$H_{\L}(T_{UV})=-T_{U^{-1}V^{-1}}=T_{VU}=-T_{UV}.$$  It follows that
$H_{\L}(N_{UV})=-N_{UV}$  and $N_{UV}$ is orthogonal to $\L$ and it
is the axis of the palindrome $UVVU-VUUV$.

\end{proof}

 Taken together the lemmas  imply

\begin{thm}\label{thm:firstPi} {\rm ({\bf The $\Pi$-map})} If $G=\langle A,B \rangle $ is a non-elementary group, there is a well
defined map $\Pi$ from any palindromic word $W$ in $G$ to the line $\L$. This map
 extends naturally
to pairs of palindromic words.  \end{thm}

\begin{proof}
If $W$ is a palindrome, set  $\Pi(W) = Ax_W \cap \L$.
If $U$ and $V$ are a pair of palindromes, set $\Pi(U,V) = N_{UV} \cap \L$.
\end{proof}

  Note that a palindrome, by its definition,  must
be cyclicly reduced. The map $\Pi$  extends to a more general set of
words obtained by cyclic permutation from the palindromes. We omit
the statement since we do not need it here.

The map $\Pi$ also extends to powers of primitive elliptics using
the following proposition.

\begin{prop} \label{lemma:Elliptic Products}

Let $E$ be an elliptic of finite order but not geometrically primitive. Let $n$ be a positive
 integer so that
that $E^n$ is geometrically primitive. If $E$ is a palindrome,  $E^n$ is  a palindrome for $n$ odd
 and the product of
palindromes for $n$ even. If $E$ is a product of palindromes, then
 $E^n$ is a product of palindromes.
 \end{prop}

 \begin{proof} Set $E=P_1P_2$. $E^2 = (P_1P_2P_1)
  \cdot P_2$  and $ E^3 = (P_1P_2P_1P_2P_1) \cdot P_2.$
  Assume  $n \ge 4$, then $(P_1P_2)^{n-1}P_1$ is a palindrome and
  $E^n = (P_1P_2)^n
   = (P_1P_2)^{n-1}P_1P_2 = [(P_1P_2)^{n-1}P_1] \cdot P_2.$
\end{proof}

Note that since the power of the elliptic may  be factored in
several ways as the product of palindromes,  in  applications it is
necessary  to be clear about what factorization is used.

\subsection{Applications }
Applying these propositions to $F$ and the primitives in $F$,  we
have the following theorem for primitive elements.

\begin{thm} {\label{thm:secondPi}} {\rm ({\bf The rational $\Pi$-map})} There is a well defined map $\Pi_{\QQ}$ from the non-negative rational numbers to
the core geodesic  $\L$ and a well defined map from $\Pr_F$, the set of primitives in $F$,
 to  $\L$.  \end{thm}

\begin{proof} Define $\Pi_{\QQ}(p/q) = \Pi(E_{p/q})$ and $\Pi_F(e_{p/q})= \Pi(E_{p/q})$.
\end{proof}

Note that we can also consider $\Pi_{|_{\Pr_G}}$, the restriction of
$\Pi$ to $\Pr_G$.

Using the maps $\phi_a,\phi_b$, we can also define $\Pi$-type maps
from the non-negative rationals to the core geodesics $\L_A, \L_B$.

\section{Discreteness Conditions}
\label{section:discreteness}

 We begin this section by recalling some basic definitions. We
refer the reader to chapter 4 of \cite{Kap} or \cite{EpM}.

If $G$ is a discrete group of isometries of $\HH^3$,  we let
$\Lambda(G)$ denote its limit set and $\Omega(G)$ its   ordinary
set. We let ${\mathcal{C}}(\Lambda(G))$ be the convex hull of the
limit set. We abbreviate these to $\Lambda$, $\Omega$ and
${\mathcal{C}}$ when the group $G$ is understood. Similarly we set
$M_G= \HH^3/G$, $\partial M_G = \Omega(G)/G$,
$N_G={\mathcal{C}}(\Lambda(G))/{G}$ and $\partial N_G=
\partial{\mathcal{C}}(\Lambda(G))$ and omit the subscript $G$ when
the groups are understood.

If $G$ is discrete without parabolics, then the convex core  $N_G$
is compact.

\begin{defn} A discrete group $G$ of isometries of $\HH^3$ is
{\sl geometrically finite} if it has a convex fundamental polyhedron
with finitely many faces.
\end{defn}

\begin{defn} A discrete group $G$ of isometries of $\HH^3$ that is geometrically
finite but contains no parabolics is called {\sl convex co-compact}.
\end{defn}

If $G$ is convex-cocompact, then $M$ is either an orbifold or a
manifold. The latter occurs when $G$ contains no elliptics.

\begin{defn} If $H \subset G$ is a subgroup, a set $C \subset \Omega$
is {\em precisely invariant} under $H$ if $h(C)=C$ for all $h \in H$
and $g(C)\cap C = \emptyset$ for all   $g \in G, g \notin H$.
\end{defn}

Suppose  that $G$ is a geometrically finite group and that   $V \in
G$ is parabolic. Let  $H=\langle V \rangle$ be the cyclic group
generated by $V$. The fixed point $p$ of $V$ is the point of
tangency of pair of open disks in the ordinary set, each of which is
precisely invariant under $H$; there is also a horoball in $\HH^3$
tangent to $\CC$ at $p$ which is precisely invariant under $H$.

\subsection{Necessary Condition for Discreteness}

\begin{thm} \label{thm:gfif} {\rm ({\bf Necessity - no parabolics})} If $G= \langle A,B \rangle$ is convex co-compact,
that is, a non-elementary discrete geometrically finite group
without parabolics, then the axes of all palindromes intersect the
core geodesic $\L$ in a compact interval.\end{thm}

\begin{proof}
If $G$ is a geometrically finite group without parabolics the convex
core $N$ of $M=\HH^3/G$ is compact.

If $G$ contains no elliptics, then we saw in \cite{GKhyp} the ends
of the
 core geodesic  $\L$ project to  two of the fixed points of the
  hyperelliptic involution of the boundary manifold
  $\partial{M}=\Omega(G)/G$.   Therefore, they lie in the regular set $\Omega$.  It follows that $\L$ intersects
  the convex hull in a compact set, and every axis lies in the convex hull.

  If $G$ contains elliptics, one can extend the results of \cite{GKhyp}
 to see that there is still an involution on the quotient $  \partial M =  \Omega(G) /G$.
 It is the projection of the half-turn about
 $\L$ and the ends of $\L$ lie in the regular set. Under the
 involution on the quotient,
   the projections of
  ends of $\L$ are  fixed points as are the projections of    the ends of $\L_A$
  and $\L_B$.
  This involution with its six fixed points is the   hyperelliptic involution on
the quotient orbifold.

If $G$ is co-compact and contains elliptics,  the argument that the
intersection of $\L$ and $\mathcal C$ is compact is the same,
provided that there is
 no elliptic element whose axis shares an end with $\L$. We claim
 this cannot happen.
 Recall (see for example, \cite{Beard}) that in a
discrete group two elements that share an axis belong either to a
cyclic subgroup $H =\langle C \rangle$ generated by an element $C$
with that axis, or to the subgroup $J=\langle C,R\rangle $ where $R$
is an order two elliptic interchanging the ends of the axis. If we
normalize so the ends are $[0,\infty]$, $C$ has the form $z
\rightarrow e^{i\theta}z$ and $R$ has the form $z \rightarrow 1/z$.
It is easy to check that the element $CR$ must be a loxodromic with
the same axis and that its ends are in the limit set. Therefore the
elliptic axis cannot share both ends with $\L$ because its ends are
not in the limit set.

If the elliptic axis shares a single end with $\L$,  normalizing so
that this end is infinity,   the subgroup of elements sharing this
end consists entirely of  affine transformations of the plane, $z
\rightarrow az+b$. Such a group contains parabolic or loxodromic
elements  implying that $\L$ has an end in the limit set. Therefore
this cannot happen.
\end{proof}

This theorem also extends to groups containing parabolic elements.
\begin{thm} \label{thm:gfifp}  \rm{{\bf  (Necessity with
parabolics)}} If $G= \langle A,B \rangle$  is a non-elementary
discrete geometrically finite group, then the axes of all
non-parabolic palindromes intersect the core geodesic $\L$ in a
compact interval.
\end{thm}

\begin{proof}  Assume $G$ is a geometrically
finite, non-elementary group containing
  parabolic elements.   By the geometric
finiteness, every abelian subgroup has rank $1$ and the convex core
has a thick and thin part.  Moreover,  each component of   the thin
part comes from  the conjugacy classes of cyclic subgroups of
parabolic elements. By Margulis' lemma \cite{Margulis,Thurston},
there is an $\epsilon$ neighborhood about the thin part that
contains no closed geodesics. The only possible axes that can
project into the core are the axes of the conjugacy class of the
parabolic(s). Any orbifold points are also outside the thin part.

  If  $P$ is a  parabolic palindrome,
its axis intersects $\L$   at one or the other end of the $\L$ line;
call such an end $x$. The thin part corresponding to the cyclic
subgroup containing $P$   lifts to a horoball at $x$ and this
horoball can contain no axes of loxodromics since there are no
geodesics in the thin part and it can contain no elliptic axes since
there are no orbifold points in the thin part.

If this is the only cyclic subgroup containing palindromic
parabolics, the other end of $\L$ is in the regular set $\Omega$ and
the convex hull boundary intersects $\L$ at a finite point.

If there is another parabolic palindrome  $Q$ not contained in the
cyclic subgroup containing $P$,  its axis and its fixed point must
be at the other end of $\L$ and $\L$ minus the intersections of $\L$
with each of the horoballs at the ends of $\L$
 is compact.  Since $\L$ has only two ends, we have accounted for all the palindromic parabolic elements
\end{proof}

\subsection{Sufficient Condition for Discreteness}
There is a stronger result in the opposite direction.  This is
independent of whether the group contains parabolic elements.

\begin{thm} \label{thm:conv} {\rm ({\bf Sufficiency})} If all palindromes in a non-elementary group
$G=\langle A,B \rangle $ have axes that intersect the core geodesic
$\L$ in a compact interval   then $G$ is discrete.
\end{thm}

\begin{proof}  In what follows, we work with elements of $SL(2,\CC)$.

We can normalize and assume the ends of $\L$ are $0$ and $\infty$.

Note $\L$ is never a point even if the axes of $A$ and $B$
intersect. Since $G$ is non-elementary $\L$ is not a degenerate
line. If the axes of $A$ and $B$ intersect in $\HH^3$, $\L$  is
defined to be the geodesic perpendicular to the axes of $A$ and $B$.

  For any non-elementary group, given any two disjoint open sets
that intersect the limit set
 there
is a loxodromic element of the group whose fixed points each lie in
one of these open sets (see \cite{Beard}).

Now suppose $G$ is not discrete and is non-elementary and suppose we
are given   $\epsilon > 0$.  Since the
   limit set is the whole plane,   we
can   find an element $C$ whose attracting fixed  point has
absolute value less than $\epsilon$.

We claim we  can also find an element $U$ both of whose
 fixed points $p,q$ have absolute values   less than $\epsilon$ and whose trace has a definite finite value.  To do this,
    choose a loxodromic  element $C$ with attracting
    fixed point very close to zero and fix an   element $D$ in $G$.

   Set $U=C^nDC^{-n}$. Its trace is equal to the trace of $D$ for any $n$ and it  satisfies the fixed point condition
    for large $n$.    Write
$U= \left(\begin{array}{cc}
         a & b \\
         c & d  \end{array}\right)$ with $ad-bc=1$. Since the trace is  a fixed finite number and the fixed points have
         absolute value less than $\epsilon$,  we deduce that
         $|c|=\mathrm O(1/\epsilon)$.

The element $V=\underleftarrow{U}$ has the form $V= -
\left(\begin{array}{cc}
         d & b \\
         c & a  \end{array}\right)$.
To see this recall that $\underleftarrow{U}^{-1} = H_L U H_L$ and
$H_L = \left(\begin{array}{cc}
i & 0\\
0 & -i \end{array}\right)$, and note that the image of $V$  in
$PSL(2,\CC)$ is the same as that of  $ \left(\begin{array}{cc}
         d & b \\
         c & a  \end{array}\right).$

         The elements $UV$ and $VU$ are palindromes by
construction. Their fixed points are, respectively
$\pm\sqrt{\frac{a}{d}}\sqrt{\frac{b}{d}}$ and
$\pm\sqrt{\frac{b}{a}}\sqrt{\frac{d}{c}}$.

 In at least one of these pairs both   fixed points lie close
to zero.

Therefore, given the compact interval $I$ in $\L$, we can choose
$\epsilon$ small enough and $U$ such that the axis of either $UV$ or
$VU$ (or both) intersect $L$ outside $I$. \end{proof}

\section{The pleating locus} \label{section:pleating}

In the case that $G$ is geometrically finite,  we want to describe
the pleating locus of the convex hull boundary. In \cite{GKhyp}
where $G$ is Fuchsian and free, we saw that the convex hull boundary
degenerates to a double of the Nielsen convex region and its
``pleating locus'' is the boundary of that region.  We described the
pleating locus in terms of certain generators and proved that it
always consisted of simple closed geodesics on the quotient.  If the
original generators for $\rho(F)$ were $A$ and $B$ where $A$ and $B$
had intersecting axes, the pleating locus was the projection of the
axis of the multiplicative commutator.  If, on the other hand, $A$
and $B$ had disjoint axes, the pleating locus consisted of the axes
determined by the so called {\sl stopping generators} $A_0,B_0$ for
the $PSL(2,\RR)$ discreteness algorithm \cite{GKwords}; namely, the
three unique simple geodesics that are the shortest geodesics  on
the surface and are the projections of the axes of $A_0$, $B_0$ and
$A_0B_0$.

 In the next section, we
recall some basic facts about the convex hull and its pleating locus
for a discrete group. References can be found in
\cite{CEpG,EpM,Kap,Thurston}. We apply some  results from
\cite{GKhyp}.

\subsection{Terminology: Pleated surfaces and the pleating locus}

Recall
\begin{defn} A {\sl pleated surface} is a pair $(X,
\mathfrak{pl})$ where $X$ is a complete  hyperbolic surface and
$\mathfrak{pl}$ is an immersion and a hyperbolic isometry of $X$
into a hyperbolic manifold $H$ such that every point in $X$ is in
the interior of some geodesic arc which is mapped to a geodesic in
$H$. The {\sl pleating locus} is the set of points in $X$ that are
contained in the interior of exactly one geodesic arc that is mapped
by $\mathfrak{pl}$ to a geodesic arc. The {\sl support of the
pleating locus} is denoted $|\mathfrak{pl}|$. The map
$\mathfrak{pl}$ is called the {\sl pleating map} and $H$ is called
the {\sl target manifold}
\end{defn}

\begin{remark} By abuse of notation we often identify $|\mathfrak{pl}|$ with
its image in $H$ and use the same symbol for it. \end{remark}

 If $G$ is a non-elementary
discrete group of isometries of $\HH^3$,
$\;{\mathcal{C}}(\Lambda(G))$ denotes the convex hull of the limit
set of $G$.  It can be constructed using {\sl support planes}. The
{\sl support planes} are hyperbolic planes in $\HH^3$ whose horizons
contain points of $\Lambda(G)$ and have the property that one of the
disks they bound contains no points of $\Lambda(G)$. The convex hull
is then the intersection of the common half-spaces bounded by the
support planes.  It is clearly $G$-invariant.

The support planes intersect in hyperbolic lines. The lines that lie
on $\partial\mathcal C$   form a geodesic lamination; that is, they
are a closed set consisting of disjoint hyperbolic geodesics. This
lamination is also called the {\em pleating locus of
$\partial\mathcal C$}. The geodesics in the pleating locus are
called pleating lines or leaves.  The pleating locus is also
$G$-invariant. See \cite{EpM} for a full discussion. The
complementary regions of the pleating locus are hyperbolic polygons
lying in a support plane with vertices on its horizon. The set of
geodesics obtained by taking the closure of the boundaries of the
complementary regions form the geodesic lamination of pleating
lines.

 If $G$ contains an elliptic
element $E$, the ends of its axis are in $\Omega(G)$ and the axis
intersects  the convex hull boundary in two points, $x$ and $y$.
Suppose a pleating line $\ell$ goes through $x$. Because $E(\ell)$
is also a pleating line and because $E(x)=x$, it follows that
$E(\ell)=\ell$ and $\ell$ is the axis of $E$.  The ends of $\ell$,
however, lie in $\Lambda$.  It follows that $x$ lies in a support
plane $T$ that is mapped to itself by $E$.  The complementary
component containing $x$ must be invariant under $E$ and, using $x$
as a vertex, can be subdivided into polygonal regions mapped to one
another by $E$.

 There is a complete hyperbolic structure
on $\partial{\mathcal{C}}$,  the boundary of the complex hull,
 defined by the hyperbolic structures on the
support planes glued together along the pleating locus. The boundary
$\partial{\mathcal{C}}$ with its pleating lamination is the image of
a pleating map by $PL$ which we will describe   in more detail
shortly. Since the support planes and pleating locus are
$G$-invariant,  the pleated surface structure and the pleating lines
project to the boundary of the quotient
$\partial{N}=\partial\mathcal C/G$ turning $\partial{N}$ into a
pleated surface with target manifold (or orbifold if $G$ contains
elliptics) $\HH^3/G$. The pleating map onto this target is the
composition of $PL$ with projection to the quotient $\pi_G: \HH^3
\rightarrow \HH^3/G$.

 We denote by $|PL|$ the support of the pleating lamination of $\partial{\mathcal{C}}$ in  $\mathbb{H}^3$
and denote by $|pl|$ its   projection to $\partial{N}$. There are
two pleating maps and pleated surfaces. Our goal is to describe the
pleating map $PL$ to $\partial{\mathcal{C}}$. The pleating map to
the quotient $\partial{N}$ is an incidental  tool that helps us
describe the pleating locus of ${\partial{\mathcal{C}}}$.

 Note that
if $G$ is geometrically finite and contains a parabolic element with
fixed point $p$, $\partial\mathcal C$ contains a pair of support
planes whose horizons are tangent at $p$ so that no (non-degenerate)
  pleating line can pass through $p$.

If the group $G$ is discrete but contains elliptics of finite order
so that the {\sl target manifold} in the construction is actually an
orbifold, we still have  a {\sl pleated surface} and use the same
terminology referring to the target as opposed to the target
manifold or obifold.

\begin{defn} We
say the pleating locus $|PL|$ (respectively $|pl|$) is {\em maximal}
if the complementary components of
  $|PL| \in \partial{\mathcal C}$ (respectively $|pl| \in \partial N$) are
 all hyperbolic triangles.
\end{defn}

\subsection{The Pleating Maps}

In order to define the pleated surface whose image is
$\partial\mathcal C$ we need a hyperbolic surface   and a pleating
map.

The boundary of the convex hull in $\HH^3$ is homeomorphic to the
regular set $\Omega=\Omega(G)$.  It has an intrinsic hyperbolic
structure defined by the hyperbolic structures on the complementary
regions glued together along their boundary geodesics. Each of the
complementary regions can be embedded in $\CC$ preserving their
hyperbolic structures and the images of the embeddings can be glued
together with their boundaries identified as determined by their
location on $\partial{\mathcal{C}}$ to form an
  infinitely connected planar domain $\tilde\Omega$,
homeomorphic to $\Omega$.  The domain $\tilde\Omega$ has a complete
hyperbolic structure whose restriction to the images of the
complementary regions agrees with their structures.  The boundaries
of the regions form a geodesic lamination on $\tilde\Omega$ with
respect to its hyperbolic structure. This is our domain surface.

Now we define  $PL$  to be the map from $\tilde\Omega$ to $\HH^3$
whose image is $\partial{\mathcal{C}} \subset \HH^3$. Thus the pair
$(\tilde\Omega,  PL)$
 is a pleated surface where the target manifold  is $\HH^3$ and $| PL|$,
 the (image of the) support of the pleating locus, is precisely the geodesic
 lamination of
 pleating lines or
 pleating locus of the group $G$.
 For any curve $\delta$ in $\partial\mathcal{C}$, we let
$\tilde{\delta}$ denote its preimage in $\tilde{\Omega}$.

  For $g \in G$, let $g_{\tilde\Omega} = PL \circ g
 \circ {PL}^{-1}$ and set  $G_{\tilde\Omega} = \{g_{\tilde\Omega} \; | \;  g \in G \}$.
 The map $PL$ induces a group  isomorphism between the groups $G_{\tilde\Omega}$ and
 $G$ so that there is a hyperbolic  isometry from   $
 \tilde\Omega/G_{\tilde\Omega}$ to  $\Omega/G$.
We look
at the universal covering $\HH^2$ of $\partial N =
\partial \mathcal{C}/G$  with covering group $\hat\Gamma$. Now
$\hat\Gamma$ is a Fuchsian group. The covering projection
$\pi_{\hat\Gamma}: \HH^2 \rightarrow
\partial N$ preserves the hyperbolic structure.
This covering factors through $\tilde\Omega$; that is since $\Omega$
and $\tilde\Omega$ are homeomorphic,  there is a map
$\pi_{\tilde{\Omega}}: \HH^2 \rightarrow {\tilde{\Omega}}$ such that
$\pi_{\hat\Gamma}   = \pi_G \circ PL \circ \pi_{\tilde\Omega}$ and
we have a group homomorphism $h$ of $\hat\Gamma$ into  $G$.

\subsection{Description of the Pleating Locus}

We now assume that $G=\langle A,B \rangle$ is a geometrically finite
group. As in section~\ref{section:discreteness},   let $\L$ be the
core geodesic, that is, the   common orthogonal to the axes of $A$
and $B$.

\begin{thm}  \label{thm:pleats} {\rm ({\bf Pleating})} Assume  that the non-elementary group $G=\langle A,B \rangle$ is
geometrically finite. Let  $\I$ be the intersection of the core
geodesic $\L$ with $\mathcal{C}$, the convex hull of the limit set
of $G$,  and let $Q_1$ and $Q_2$ be the endpoints of $\I$. Let
$(\partial{\mathcal{C}}, PL)$ be the corresponding pleated surface.
Then one of the following occurs
\begin{enumerate}
\item There is a  leaf of $|PL|$, the support of the image of the pleating locus of $\partial\mathcal C$,
 intersecting  $\I$
orthogonally at one   of its endpoints and this leaf is an
accumulation of axes of palindromes. This leaf may intersect
$\I$ at both $Q_1$ and $Q_2$ or there may be two such distinct
leaves one intersecting $\I$ at $Q_1$ and the other ar $Q_2$.

\item  Every leaf of $|PL|$ is disjoint from $\I$ and $Q_1$ and
$Q_2$ lie in complementary components   of $\partial{\mathcal
C}$. In this case the pleating is not maximal.
\end{enumerate}
\end{thm}

\begin{proof}  We proved in \cite{GKhyp} that if $G$ is Fuchsian,  then
$H_{\L}$, the half turn about $\L$ projects to the hyperelliptic
involution on $\partial{M}$  and  the ends of $\L$ project to
Weierstrass points on $\partial{M}$ so that they are in the ordinary
set. We also proved that $H_{\L}$ induces the hyperelliptic
involution on $\partial N$.  That proof extends in a straightforward
manner to any geometrically finite $G$.

 We proceed
through several steps:

\begin{enumerate}
\item We note that $H_{\L}$ fixes the limit set and thus it fixes
the convex hull and the boundary of the convex hull. We claim it
also fixes the pleating locus.

\noindent {\it Proof of claim:}  Let $\ell$ be a leaf of the
pleating locus and let $x$ and $y$ denote its ends. Suppose $\L$
intersects a leaf $\ell$ in the pleating locus. Then
$H_{\L}(\ell) \cap \ell \ne \emptyset$;  that is, the image of
the leaf under the half-turn will intersect itself at the point
of intersection with $\L$. But leaves in the pleating locus are
disjoint so $H_{\L}(\ell) = \ell$  and $\ell$ is orthogonal to
$\L$.

Suppose $H_{\L}(\ell)$ does not intersect $\ell$ and suppose
$\ell$ is the intersection line of two support planes. Since
$H_{\L}$ maps support planes to support planes, it will preserve
intersections of support planes and $\H_{\L}(\ell)$ is a
pleating line.  If $\ell$ is a limit of pleating lines $\ell_n$
that are intersections of support planes,
 then $H_{\L}(\ell_n)$ are pleating lines and, since the
 pleating lines form a closed set, their limit $H_{\L}(\ell)$ is
 also a pleating line.

\item

Assume there is a leaf $\ell$ of $|PL|$ passing through one of
the points $Q_1$ or $Q_2$.   Since $|PL|$ is invariant under
$H_{\L}$ and the leaves are disjoint, $\ell$ is orthogonal to
$\L$.  If $\ell$ is the axis of an element of $G$, and thus
projects to a simple closed curve on $\partial N$, the element
must be a palindrome because its axis is orthogonal to $\L$ and
we are done. (Note the projection is simple because leaves
project to leaves and these are simple.)

 \item If $\ell$ is not the axis of an element we have to argue
further. We need to find a sequence of palindromic elements
  whose axes converge to $\ell$. We
cannot do this directly because, in general, axes of palindromic
elements are not leaves of the pleating.

\item
 The idea of the proof is to find a sequence of closed
approximating leaves to $\pi_G(\ell)$ on $\partial N$.  These leaves
lift to a sequence of axes   of elements of $\hat\Gamma$ that limit
on a lift of  $ \pi_G(\ell)$ in $\HH^2$ by $\pi_{\hat\Gamma}^{-1}$.
  Because
$\hat\Gamma$ is discrete, the traces of the elements in this
sequence diverge to infinity.

We then look at the images of these axes in $\mathcal C$ under
the map  $  PL \circ \pi_{\tilde\Omega}$.  Since their ends
converge to the ends of $\ell$,  they are ends of axes of
non-trivial elements of $G$ and since $G$ is discrete,  the
traces of these elements also diverge to infinity. We use this
sequence of elements in $G$ to construct a sequence of
palindromes in $G$ whose axes converge to $\ell$.
\item
In order to find the appropriate lifts, recall that  the axes of
$A$ and $B$ lie in $\mathcal{C}$ and are invariant under $H_\L$.
Using a variant of the nearest point retraction, we can find a
curve $\alpha$ invariant under $A$ and a curve  $\beta$
invariant under $B$ both lying on $\partial{\mathcal{C}}$ and
these curves can be chosen to be  invariant under $H_{\L}$. In
the same way, we can also find a closed curve $\lambda$ on
$\partial\mathcal C$ passing through $Q_1$ and $Q_2$ invariant
under $H_{\L}$.

  The map ${\widetilde{H_{{\L}}}}= {PL}^{-1} H_{\L} PL$ is an
involution on $\tilde\Omega$. The curves $\tilde{\alpha},$
$\tilde{\beta},$
and $\tilde{\lambda}$
are invariant under ${\widetilde{H_{\L}}}$.  Denote the
pre-images of the points $Q_1, Q_2$ by $\tilde{Q_1},\tilde{Q_2}$

We now choose elements $\hat{A},\hat{B}$ of $\hat\Gamma$ whose
axes project to curves in the free homotopy classes of
$\tilde\alpha, \tilde\beta$ and have the additional property
that their common orthogonal $\hat\L$ projects to
$\tilde\lambda$.  The involution ${\widetilde{H_{{\L}}}}$ lifts
via $\pi_{\tilde{\Omega}^{-1}}$  to $\HH^2$ where it is
reflection in the  line $\hat\L$.

\item Assume for definiteness that the leaf $\ell$ passes through $Q_1$ and that
the axis of $A$ intersects $\L$ between $Q_1$ and its intersection
with the axis of $B$.  Let $\tilde\ell=PL^{-1}(\ell)$; then it
passes through $\tilde{Q_1}$.  Let $\hat{Q_1}$ be the closest
pre-image of $\tilde{Q_1}$ under $\pi_{\tilde{\Omega}^{-1}}$ to the
intersection of the axis of $\hat{A}$ and $\hat\L$ on the far side
of the intersection of $\hat{\L}$ and the axis of $\hat{B}$.

The projection $\gamma=\pi_G(\ell)$ is a simple open infinite
curve on $\partial N$
   in the closure of the set of simple closed geodesics on
$\partial N$ in the Thurston topology on these curves. We can
therefore choose a sequence of simple closed geodesics
$\gamma_n$ on $\partial N$ converging to $\gamma$ and a sequence
of lifts $\hat{\ell_n}$ in $\HH^2$ converging to $\hat\ell$.
These lifts are axes of elements of $\hat\Gamma$ that we denote
by $\hat{W}_n$. Because the axes converge to $\hat\ell $, for
$n$ large they intersect $\hat{L}$, but they do not necessarily
intersect orthogonally. Since they are converging to $\hat\ell$
which is orthogonal to $\hat{L}$, the angles of intersection
tend to a right angle.

\item Consider the images
$PL \circ \pi_{\tilde\Omega}(\hat{\ell_n})=\ell_n$ in $\partial
\mathcal C$. Because their axes converge to $\hat\ell$,   the
corresponding words $W_n = h(\hat{W}_n) \in G$ are non-trivial.
 If $W_n$ is not a palindrome, then $W_n$ and $H_{\L}W_nH_{\L}$
have distinct axes and their common normal intersects $\L$.

We claim we can make our choice of $\gamma_n$ so that infinitely
many of the $W_n$ are palindromes.  We do this by constructing a
palindrome $P_n$ from each $W_n$ that is  not a palindrome and
replacing the non-palindromic $W_n$ by the $P_n$ so constructed.
\item Suppose  $L=[0,\infty]$ and $\gamma=[-r,r]$. Set
$$W_n =\left(\begin{array}{cc} a_n & b_n \\ c_n & d_n \end{array}\right)$$ where
$a_nd_n-b_nc_n=1$. Then, as we saw in
section~\ref{section:discreteness},
$${\underleftarrow{W_n}}= \left(\begin{array}{cc} d_n & b_n \\ c_n &
a_n
\end{array}\right).$$

Now
$$W_n{\underleftarrow{W_n}}= \left(\begin{array}{cc} a_nd_n+b_nc_n &
2a_nb_n \\ 2c_nd_n & a_nd_n+b_nc_n \end{array}\right).$$

By the construction of $W_n$, the fixed points of $W_n$  tend to
the ends $\pm r$ of $\gamma$. Since $G$ is co-compact, $r$ is
finite.

The fixed points of $W_n$ are $$\frac{(a_n-d_n) \pm
\sqrt{(a_n-d_n)^2-4b_nc_n}}{2c_n},$$  the fixed points of
${\underleftarrow{W_n}}$ are just the rotations of the fixed
points of $W_n$ by $\pi$; namely,
$$\frac{(d_n-a_n) \pm
\sqrt{(a_n-d_n)^2-4b_nc_n}}{2c_n}.$$ The fixed points of the
palindromes  $P_n = {\underleftarrow{W_n}}W_n$ are
$$\pm \sqrt{\frac{a_nb_n}{c_nd_n}}.$$

Since the axes of $W_n$ approach a pleating line, the fixed
points of $W_n$ tend to $\pm r$ and the discreteness of the
group implies the traces $(a_n+d_n)$ tend to infinity.

It follows that either or both $|a_n|$ and $|d_n|$ go to infinity,
$(a_n-d_n)/2c_n \rightarrow 0$, $r=\lim_{n\to \infty}
\sqrt{b_n/c_n}$.  We claim $\lim a_n/d_n = 1$.

Suppose first that $\lim c_n = c$ with $c  \neq \infty$; then,
since $(a_n-d_n)/2c_n \rightarrow 0$, $\lim a_n -d_n =0$ and
$\lim a_n/d_n = 1 $. If $\lim c_n =\infty$ and  $\lim a_n - d_n
= f$ where $f  \neq \infty$, then again $\lim a_n/d_n - 1 = \lim
f/d_n = 0$ or $\lim 1-d_n/a_n= \lim f/a_n = 0$.

We deduce that $$\lim \pm \sqrt{\frac{a_nb_n}{c_nd_n}}= \lim\pm
\sqrt{\frac{ b_n}{c_n }}= \pm r.$$

We conclude that the axes of the palindromes $P_n$ intersect
$\L$ in a sequence of points that tend to $Q_1$.  The axes must
accumulate in the plane perpendicular to $\L$ at $Q_1$ and their
ends tend to the ends of $\ell$ as required.

\item  Suppose now that no
pleating line passes through   the point $Q_1$  and that it is
contained in a complementary region  $\sigma$ of $|PL|$.   Since
$Q_1$ is fixed by $H_{\L}$, and $H_{\L}$ preserves the pleating
lines and complementary regions, $H_{\L}(\sigma)=\sigma$. Therefore,
if $\ell$ is one leaf of $\partial\sigma $, $H_{\L}(\ell)$ is also a
leaf of $\partial\sigma$.  Since no boundary leaf passes through
$Q_1$, no boundary leaf is invariant under $H_{\L}$ and the boundary
leaves occur in pairs; thus the region cannot be a triangle.
\end{enumerate}\end{proof}

\bibliographystyle{amsplain}

\end{document}